\documentclass[a4paper,reqno]{amsart}
 \usepackage{amssymb,amsmath,amscd,graphicx,color,epstopdf,mathtools}
\usepackage[all]{xy}



\newcommand{\length}{\operatorname{length}}
\newcommand{\cF}{\mathcal{F}}

\newcommand{\cH}{\mathcal{H}}
\newcommand{\cL}{\mathcal{L}}
\newcommand{\w}{\omega}
\newcommand{\R}{\mathbb{R}}

\newcommand{\Z}{\mathbb{Z}}

\newtheorem{proposition}{{\bf Proposition}}
\newtheorem{lemma}{{\bf Lemma}}
\newtheorem{theorem}{{\bf Theorem}}

\newtheorem{definition}{ {\bf Definition}}

\begin{document}
\author{David Mart\'inez Torres}
\address{Departamento de Matem\'atica, PUC-Rio, R. Mq. S. Vicente 225, Rio de Janeiro 22451-900, Brazil}
\email{dfmtorres@gmail.com}

\author{\'Alvaro del Pino}
\address{Universidad Aut\'onoma de Madrid and Instituto de Ciencias Matem\'aticas -- CSIC.
C. Nicol\'as Cabrera, 13--15, 28049, Madrid, Spain.} 
\email{alvaro.delpino@icmat.es}

\author{Francisco Presas}
\address{Instituto de Ciencias Matem\'aticas -- CSIC.
C. Nicol\'as Cabrera, 13--15, 28049, Madrid, Spain. }
\email{fpresas@icmat.es}

\title{The foliated Lefschetz hyperplane theorem}

\begin{abstract} A foliation $(M,\cF)$ is said to be $2$--calibrated if it admits a closed 2-form $\w$
making each leaf symplectic. By using approximately holomorphic techniques, a sequence $W_k$ of $2$--calibrated submanifolds
of codimension--$2$ can be found for $(M, \cF, \w)$. Our main result says that the Lefschetz hyperplane theorem holds
for the pairs $(F, F \cap W_k)$, with $F$ any leaf of $\cF$. This is applied to draw important consequences on the transverse geometry
of such foliations.
 \end{abstract}

\maketitle

\section{Introduction and statement of the main result}

A foliation $\cF$ by surfaces on a $3$--dimensional closed manifold $M$ is called taut if for every leaf there exists a loop through it that is everywhere transverse to $\cF$. This topological definition is equivalent to the following differential geometric one: there exists a closed 2-form inducing an area form on each leaf \cite{Su79}. Tautness implies strong topological restrictions on the pair $(M, \cF)$: by the work of Novikov we know that the fundamental group of any leaf injects into the fundamental group of the ambient manifold, and that every loop $C\pitchfork \cF$ must be non--trivial in homotopy.

The straightforward generalization of tautness to arbitrary dimension requires the existence, through any leaf, of a loop everywhere transverse to the foliation. However, in deep contrast to the $3$--dimensional case, these objects are quite flexible, as shown by the $h$--principle proved by Meigniez \cite{Me12}. In \cite{IM04a}, the first author proposed the following alternative generalization of taut foliations to higher dimensions:

\begin{definition} \label{def:main} A codimension--$1$ foliation $\cF^{2n}$ of $M^{2n+1}$ is said to admit a 2-calibration if there exists a closed $2$-form $\omega$ such
that the restriction of $\omega^n$ to the leaves of $\cF$ is nowhere vanishing. A triple $(M,\cF,\w)$, where $\w$ is a 2-calibration for $\cF$, is referred to as
a 2-calibrated foliation.
\end{definition}

\begin{definition}
A submanifold $W\hookrightarrow (M,\cF,\w)$ is a \emph{2--calibrated submanifold} if it is everywhere transverse to $\cF$ and it intersects each leaf of $\cF$ in a symplectic submanifold w.r.t. $\w$.
\end{definition}

As in the symplectic and contact settings, Donaldson's approximately holomorphic techniques \cite{Do96} can be applied to study 2-calibrated foliations.
In particular, they can be used for the construction of \textsl{2--calibrated divisors}:

\begin{proposition} \cite[Corollary 1.2]{IM04a}
Let $(M^{2n+1}, \cF; \w)$ be a $2$--calibrated foliation on a closed manifold with $\w$ of integral class. Then, for any integer $k$ large enough, there are $2$--calibrated submanifolds $W_k^{2n-1}$ representing the Poincar\'e dual of $[k\w]$.

Additionally, the maps
\[ i_*: \pi_j(W_k) \to \pi_j(M) \]
\[ i_*: H_j(W_k, \Z) \to H_j(M, \Z) \]
are isomorphisms for $j<n-1$ and surjections for $j=n-1$.
\end{proposition}

The submanifolds $W_k$ in the proposition will be called \textsl{Donaldson divisors}. The second part of the statement is the \textsl{$2$--calibrated Lefschetz hyperplane theorem}: much like in the projective and the symplectic cases, the divisors recover some of the topology of the ambient. The purpose of this note is to prove the following analogous result:

\begin{theorem}\label{thm:main}
 Let $(M^{2n+1},\cF,\w)$ be a $2$--calibrated foliation on a closed manifold. Let $W$ be a Donaldson divisor of dimension $2n-1$. Then, for every leaf $F$ of $\cF$ it holds that\footnote{For $A\subset B$, we have $\pi_0(B, A)=\pi_0(B)/\pi_0(A)$. From the definition, this extends the long exact sequence for the pair to the $\pi_0$-level, see  \cite{Hat} pag. 476.}
 \[\pi_k(F,F\cap W)=\{1\},\,0\leq k \leq n-1.\]
\end{theorem}

This result says that, despite the fact that a leaf $F$ of $\cF$ might be non--compact,
the symplectic Lefschetz hyperplane theorem holds for the pair $(F, F \cap W_k)$. It should
be remarked that the case of $\pi_0$ was already proved in \cite{IM04a,Ma12}, where the question was tackled constructing $2$--calibrated Lefschetz pencils.
In fact, it is clear that the relative $\pi_1$ can be computed out of \cite{Ma12}. 
The method of proof in this note is new, not an adaptation of the Lefschetz pencil techniques, and yields a much shorter and simpler proof for the $\pi_0$ and the  $\pi_1$ case as well.

The following theorem states an important consequence of Theorem \ref{thm:main}. 

\begin{theorem} \label{thm:applications}
Let $(M^{2n+1}, \cF, \w)$ be a $2$--calibrated foliation on a closed manifold.
Then there exists a closed 3-dimensional $2$--calibrated submanifold $(W,\cF_W = W \cap \cF, \w|_W) \hookrightarrow M$ satisfying 
the following equivalent properties:
\begin{enumerate}
\item[(i)] the map between holonomy groupoids induced by the inclusion 
\[ \iota\colon \mathrm{Hol}(\cF_W)\rightarrow\mathrm{Hol}(\cF) \]
is an essential equivalence;
\item[(ii)] any total transversal $T$  for $(W,\cF_W)$ is also a total transversal for $(M,\cF)$,
and the holonomy pseudogroups $\cH(\cF,T)$ and $\cH(\cF_W,T)$, induced on $T$ by $\cF$ and $\cF_W$, respectively, coincide.
\end{enumerate}
\end{theorem}
\begin{proof}
 This follows by first observing that, if $W$ is a submanifold transverse to $\cF$, 
  then, for either condition (i) or (ii) to hold (see \cite{MM03} for background material on essential equivalences and holonomy groupoids),
  it suffices that for each leaf $F\in \cF$
 \[\pi_0(F,F\cap W)=\pi_1(F,F\cap W)=\{1\}.\]

Theorem \ref{thm:main} can be applied as long as $n>1$. Doing so iteratively yields a descending chain of Donaldson divisors $M^{2n+1} = W_0 \supset W^{2n-1}_1 \supset \cdots \supset W^3_{n-1}$ satisfying 
\[ \pi_0(F \cap W^{2(n-k)+1}_k, F\cap  W^{2(n-k)-1}_{k+1}) = \{1\}, \]
and
\[ \pi_1(F \cap W^{2(n-k)+1}_k, F\cap  W^{2(n-k)-1}_{k+1}) = \{1\}. \]
This proves the claim. 
\end{proof}

Note that a 3-dimensional 2-calibrated submanifold as in Theorem \ref{thm:applications} is a classical 3-dimensional taut foliation. The map $\iota\colon \mathrm{Hol}(\cF_W)\rightarrow\mathrm{Hol}(\cF)$ being an essential equivalence implies not just that the map induced on leaf spaces $W/\cF_W\to M/\cF$ is a homeomorphism (c.f. \cite{Ma12}), but that both foliations have the same \textsl{transverse geometry}.

To spell this out more precisely, this implies in particular that:
\begin{itemize}
 \item The homeomorphism on leaf spaces preserves the growth type of the leaves \cite{Ha02}.
 \item There is a bijection between the \textsl{transverse geometric structures} on $(W,\cF_W)$ and those on $(M,\cF$). These are, for instance: holonomy invariant 
 transverse (Radon) measures, Riemannian metrics -- in general, structures defined by (invariant) sheaves over the Haefliger groupoid $\Gamma^1_\infty$ --, and real analytic
 structures (i.e. reductions to $\Gamma^1_{\omega}$).
 \item There is an isomorphism between  the periodic, Hochschild and periodic cyclic homologies of the convolution algebra of the holonomy groupoids \cite{CM}.
\end{itemize}

\textbf{Acknowledgments.} The authors are grateful to M. Crainic and P. Frejlich for their valuable suggestions.
The present work is part of the authors activities within CAST, a Research Network Program of the European Science Foundation. 
The first author acknowledges partial support of the FCT 2007 Ciencia Program and ERC Starting Grant no. 279729.
The second and third authors are supported by the Spanish National Research Project MTM2013---42135 and by the ICMAT Severo Ochoa grant SEV-2015-0554.

\section{Ingredients of the proof}

Our proof of Theorem \ref{thm:main} follows Donaldson's proof of the Lefschetz hyperplane theorem for approximately holomorphic divisors. His proof followed the Andreotti-Frankel proof in the affine/projective case: in the complement of a divisor,  the modulus of its defining approximately holomorphic section can be regarded, after a small perturbation, as a Morse function with critical points of index at least $n$, which implies that the ambient manifold is obtained from the divisor by attaching handles of index at least $n$. It readily follows that the relative homology and homotopy groups of degree less than $n$ vanish. 

In the foliated case however, the critical points of this function come in $\mathbb{S}^1$ families and a non--compact leaf will, in general, have infinitely many critical points. Hence, the relative homotopy type of a leaf with respect to the divisor is not readily understood. 

We shall first review the essentials of the approximately holomorphic machinery that we will need for the proof of Theorem \ref{thm:main}. Then we will discuss some conditions for Morse functions in open manifolds that will guarantee a nice behavior for their gradient flow.

\subsection{The approximately holomorphic theory for 2-calibrated foliations} \label{sec:ah}

Let $M^{2n+1}$ be a closed manifold endowed with a 2-calibrated foliation $(\cF,\w)$. After a small perturbation, 
we may assume without loss of generality that $[\w]$ is a rational class; by scaling the class, we may also assume
that it is integral. We let $\mathcal{\cL}\rightarrow M$ be the pre--quantum line bundle associated to $\w$; 
this is a Hermitian line bundle with a compatible connection $\nabla$ whose curvature is $-2\pi i\w$.

We let $\nabla^\cF$ denote the component of $\nabla$ tangential to $\cF$.
After choosing an almost complex structure $J$ compatible with $\w$, the tangential connection can be further
decomposed into its complex linear and antilinear parts, yielding $\nabla^\cF = \partial + \bar\partial$.

According to \cite{IM04a}, Corollary 1.2, upon choosing the almost complex structure $J$, it is
possible to construct a family $s_k: M \to \cL^k$ of sections of the $k$-th tensor powers of $\cL$,
for $k$ large enough, such that $W_k:=s_k^{-1}(0)$ are closed, 2-calibrated submanifolds of codimension two.

To state the conditions that are required for the sequence $s_k$,
we fix a metric $g$ on $M$ which over the leaves satisfies $g=\w(\cdot,J\cdot)$. 
Further, we define a family of scaled metrics $g_k = kg$.
\begin{definition}  \hfill
\begin{enumerate}
\item A sequence of sections $s_k: M \to \cL^k$ is said to be approximately holomorphic if there is a universal constant $C > 0$ such that:
\[ |s_k|_{g_k}, |\nabla s_k|_{g_k} < C ; \qquad |\bar\partial s_k|_{g_k}, |\nabla \bar\partial s_k|_{g_k} < Ck^{-1/2}, \]
for $k$ large enough.

\item A sequence of sections $s_k: M \to \cL^k$ is said to be $\nu$-transverse to zero along
the foliation $\cF$ if at any point either $|s_k|_{g_k} \geq \nu$ or $|\nabla^\cF s_k|_{g_k} \geq \nu$.
\end{enumerate}
\end{definition}

To every such an approximately--holomorphic transverse to zero sequence $s_k$ one associates a sequence
of functions $f_k: M \setminus W_k \to \mathbb{R}$ by $f_k = \log|s_k|^2$. The Lefschetz hyperplane theorem 
for Donaldson-type submanifolds (\cite{Do96,IM04a}) states:
\begin{proposition} \label{propo:index}
Fixing a leaf $F$, the function $f_k: F \setminus (W_k \cap F) \to \mathbb{R}$, which might not be Morse,
has only critical points of index at least $n$.
\end{proposition}

This proposition, when applied to a closed leaf, implies Theorem \ref{thm:main} immediately --as seen in \cite{Do96, Ma12}-- 
 for Donaldson divisors.

\subsection{Gradient flows and the topology of open manifolds.}

The study of flows which behave well on  open manifolds already 
appears in the literature on foliation theory \cite{FW86}. For the sake of completeness, we review these facts tailored
to the applications we have in mind. 

Let $f$ be a Morse function on a manifold $M$. For any $a\in \R$  set $M_a=\{x \in M\,|\,f(x)\leq a\}$, and denote by $\mathrm{Crit}_a(f)$
the subset of critical points of $f$ lying in $M\backslash M_a$.

Let $a$ be a regular value for $f$ and let $b>a$. Assume for the moment that $M$ is compact. 
It is customary to study the relative topology of the pair $(M_b,M_a)$  using minus the gradient flow of $f$ 
with respect to some fixed metric $g$.
The key point is that the following dichotomy holds: for any $x\in M_b\backslash M_a$
the trajectory of $-\nabla_gf$ starting at $x$ either enters $M_a$  in finite time, or converges to one of 
the finitely many critical points in
$\mathrm{Crit}_a(f)$.  

If $M$ is no longer compact but $f$ is proper, then of course the study of  the relative topology of the pair $(M_b,M_a)$
goes exactly as in the
compact case. There might be cases --as in our setting coming from approximately holomorphic geometry-- that the natural Morse functions
to be used
are not proper, and one needs to impose an appropriate form of the above dichotomy for trajectories of $-\nabla_gf$:

\begin{lemma}\label{lem:dich} Let $f$ be a Morse function on a manifold $M$ and let $g$ be a metric on $M$ so that $\nabla_gf$ is complete. 
Let
$a$ be a regular value, $b>a$, and assume that the following holds:
 \begin{enumerate}
  \item For every compact subset $X\subset M_b$, there exist finitely many critical points $c_1,\dots,c_{i_X}$ in $\mathrm{Crit}_a(f)$ 
   such that the following dichotomy holds:  a trajectory of $-\nabla_gf$ starting at $x\in X$ either reaches $M_a$ in finite time,
  or converges to a critical point in $\{c_1\}\cup\cdots\cup \{c_{i_X}\}$. 
  \item Every $c \in Crit_a(f)$ has index $\geq j$. 
 \end{enumerate}
Then we have that $ \pi_k(M_b, M_a) = 0$, for $k=0, \ldots j-1$.
\end{lemma}

\begin{proof} 
Let us start by making the following observation: if $X$ is as in assumption $(1)$ and the collection 
$\{c_1\}\cup\cdots\cup \{c_{i_X}\}$ is empty, then we claim that $X$ is taken in finite time to $M_a$ 
by the flow $\phi$ of $-\nabla_gf$. Indeed, for every $x \in X$ there exists a time $t_x > 0$ such that 
$f(\phi_{t_x}(x)) < a$; further, since for fixed $t$, $\phi_t$ is continuous, there is a small ball 
$B_g(x,\varepsilon_x)$ centered at $x$ such that $\phi_{t_x}(B_g(x,\varepsilon_x)) \subset M_a$. 
Then, the result follows by compactness of $X$. 

Now, let $N$ be a compact manifold and $h\colon (N,\partial N) \rightarrow (M_b, M_a)$ be a smooth map. Let $U$ be a relatively
compact neighborhood of $h(N)$. Then assumption $(1)$ implies that trajectories starting at points in $\bar{U}$ can only enter
$M_a$ in finite time or converge to one of the finitely many critical points $\{c_1,\dots,c_{i_{\bar{U}}}\}$. 

Observe that there is a small relatively compact neighborhood $V$ of $h(\partial N)$ such that the flow 
of $-\nabla_gf$ sends $V$ into $M_a$: this follows if $V \subset U$ is selected so that $f(V)$ lies below the critical
values $\{f(c_1),\dots,f(c_{i_{\bar{U}}})\}$.

We now construct $h'$, an arbitrarily small perturbation of $h$ relative to $V$. Proceeding inductively over the finite list $\{c_1,\dots,c_{i_{\bar{U}}}\}$, as in \cite{Mil}, we obtain $h'$ that is transverse to the ascending disks of the critical points and that satisfies $h'(N) \subset U$.

If $N$ has dimension at most $j-1$ then, by hypothesis $(2)$, transversality to the ascending disks means 
empty intersection. The hypotheses of the claim at the start of the proof are satisfied and it follows that $ \pi_k(M, M_a) = 0$, for $k=0, \ldots j-1$.
\end{proof}

The following result describes quantitative conditions on the gradient vector field granting the dichotomy in point (1) of Lemma \ref{lem:dich}.

\begin{proposition}\label{pro:dich}
 Let $f$ be a Morse function, $g$ be a complete metric on $M$, and $a<b\in \R$. Assume that there exist real constants $D,E>0$ and open subsets $\mathcal{C}_i \subset M_b$, $i\in I$, such that:
 \begin{enumerate}
\item For any pair $i,{i'}\in I$, $i\neq i'$,  we have $d_g(\mathcal{C}_i,\mathcal{C}_{i'}) > D$.
\item The diameter of the sets $\mathcal{C}_i$ is at most $E$.
\item There exist real numbers $\delta_1, \delta_2 > 0$, such that  
\[\delta_2 \geq|\nabla_g(f)(p)| \geq \delta_1, \forall p \in M_b \setminus \left(\bigcup_{i\in I} \mathcal{C}_i \right). \]
\end{enumerate}

Then $-\nabla_gf$ is complete and the dichotomy in point (1) of Lemma \ref{lem:dich} for $-\nabla_gf$ holds. 
\end{proposition}

Essentially, the proposition states that the critical points of $f$ come in families, indexed by $I$ and contained in the sets $\mathcal{C}_i$, that are far from each other. In order to prove Proposition \ref{pro:dich}, let us introduce some notation and prove an auxiliary lemma. Given any $x \in M$, we denote by $\gamma_x$ the positive half of the flow line that contains $x$. Denote by $\phi_t$ the flow of $f$ at time $t$. Let $\gamma_x^t$ designate the segment of the curve $\gamma_x$ between $x$ and $\phi_t(x)$. Then:

\begin{lemma} \label{lem:appendix}
Under the assumptions of Proposition \ref{pro:dich}, there is a constant $R$, independent of $t\in \mathbb{R}$ and $x \in M_b$, such that $d_g(\phi_t(x), x) > R$ implies $f(\phi_t(x)) < a$.
\end{lemma}
\begin{proof}
For every curve $\gamma$ we denote by $\widetilde \gamma$ the (possibly disconnected) curve:
\[ \widetilde \gamma = \left\{ p \in \gamma: p \notin \bigcup_{i\in I} \mathcal{C}_i \right\}, \]
that is, the union of segments of $\gamma$ that are disjoint from the sets $\mathcal{C}_i$.

Given any curve $\gamma \subset B(x,R)$ starting at $x$ and intersecting the boundary of $B(x,R)$ at $y$, we can associate to it another curve, which we denote by $\eta = \eta_\gamma$, using the following procedure:
\begin{enumerate}
\item list, in order, all the sets $\mathcal{C}_i$ that $\gamma$ intersects. Remove all the consecutive repetitions of the same $\mathcal{C}_i$, listing just the first one in each series of repetitions. Write $\{ \mathcal{C}_{i_j}\}_{j \in [1,..k]}$ for this finite list,

\item mark the entry and exit points $e_j$ and $f_j$ of $\gamma$ into each $\mathcal{C}_{i_j}$. In the case of consecutive repetitions of the same $\mathcal{C}_i$, just mark the first entry point and the last exit point of the series. For simplicity, denote $f_0 = x$ and $e_{k+1} = y$,

\item call $\eta$ the piecewise smooth curve formed by connecting these marked points in the order they appear. From $e_j$ to $f_j$, take the shortest geodesic between the two points. From $f_j$ to $e_{j+1}$, take the shortest path not intersecting any $C_i$. Denote these paths by $l(e_j,f_j)$ and $l(f_j,e_{j+1})$ respectively. 
\end{enumerate}

Assume $R > E+D$. If $k=0,1$, it is immediate that 
$$ \dfrac{\length(\widetilde\eta)}{\length(\eta)} \geq \frac{D}{E+D},  $$
otherwise, the following estimate holds:
\[  \dfrac{\length(\widetilde\eta)}{\length(\eta)} = \dfrac{\sum_{j=0}^k \length(l(f_j,e_{j+1}))}{\sum_{j=0}^k \length(l(f_j,e_{j+1})) +  \sum_{j=1}^k \length(l(e_j,f_j))} \geq \]
\[ \dfrac{\sum_{j=1}^{k-1} \length(l(f_j,e_{j+1}))}{\sum_{j=1}^{k-1} \length(l(f_j,e_{j+1})) +  kE} \geq \dfrac{(k-1)D}{(k-1)D+kE} \geq \dfrac{D}{2(E+D)} .\] 

For any radius $r > E+D$, denote by $\tau$ the time at which the curve $\gamma_x$ first intersects $\partial B(x,r)$. Denote this intersection point by $y$. Consider the segment $\gamma_x^\tau$ and its associated curve $\eta = \eta_{\gamma_x^\tau}$. Use the fact that over $\widetilde \gamma_x^\tau$ we have a lower bound for the gradient  $|\nabla_g f| > \delta_1 > 0$:
\[ |f(y) - f(x)| \geq \delta_1 \length(\widetilde\gamma_x^\tau) \geq \delta_1 \length(\widetilde\eta) \geq \delta_1 \length(\eta) \dfrac{D}{2(E+D)} \geq r\dfrac{\delta_1 D}{2(E+D)} \]
which implies that, if $r$ is taken to be large enough, $|f(y) - f(x)| > b-a$, and hence $y \in M_a$.
\end{proof}

\begin{proof}[Proof of Proposition \ref{pro:dich}]
Let $X \subset M_b$ be a compact set. Let $R$ be the universal constant given by Lemma \ref{lem:appendix}. 
Denote by $X(R)$ the $R$-neighborhood of $X$, which is a relatively compact set.
Lemma \ref{lem:appendix} implies that any trajectory starting at $X$ either reaches the interior of $M_a$ -- 
which is equivalent to saying that it reaches $M_a$ in finite time -- or it remains in $X(R)$ for all time. 

It must be shown that if a trajectory $\gamma_x$ remains within $X(R)$ for all times then it must converge to a critical point.
Since $X(R)$ is relatively compact and $f$ is a Morse function, there is a finite number $k$ of critical points in its closure.
Each of those critical points $\{c_i\}_{i=1}^{k}$ has an arbitrarily small neighborhood $V_i$ which corresponds to a ball 
in the standard Morse model around $c_i$. In particular, a trajectory that intersects $V_i$ must intersect just once, 
either converging to $c_i$ or escaping from $V_i$ eventually. From this it follows that there is a time $t_0 > 0$ such
that $\gamma_x(t) \notin V_i$, for all $t > t_0$ and every $i$. Since the gradient $|\nabla_g f| > \delta > 0$ is
bounded from below in $X(R) \setminus \cup_{i=1..k} V_i$, this shows that $f(\gamma_x(t)) < a$ for $t$ large enough, which is a contradiction.
\end{proof}

\section{Proof of Theorem \ref{thm:main}.}

Fix some leaf $F \in \cF$. All we need to do now is to check that, for a suitable choice of Morse
function and metric on $F$, the hypotheses of Proposition \ref{pro:dich} 
are satisfied for $F$. Our candidate
is the restriction to the leaf of the function $f_k = \log|s_k|^2$, and the restriction to the leaf of any Riemannian metric
on  $M$. 

 We shall prove a couple of preliminary lemmas, for which we need to recall some notation. 
 Given a function $f$, defined on a manifold endowed with a codimension one foliation $(M, \cF)$,
 the tangential differential $d^\cF f$ is the composition of the differential with the projection $T^*M\rightarrow (T\cF)^*$. 
 The points in which $d^\cF f$ vanishes are the \emph{tangential critical points of $f$}, which we denote by $\Sigma^\cF(f)$. Of course, $\Sigma^\cF(f)$ are nothing but the critical points of the restriction of $f$ to each leaf of $\cF$. 

\begin{lemma} \label{lem:neigh}
For every $k$ large enough, the $2$--calibrated submanifold $W_k \subset M$ has a tubular neighborhood that contains a full regular level set
of $f_k = \log|s_k|^2$ and which is also disjoint from $\Sigma^\cF(f_k)$.
\end{lemma}
\begin{proof}
It is enough to check that $h_k = ||s_k||^2$ satisfies the Lemma, since $\log$ is an increasing monotone function. 

 We claim that the neighborhood $U=\{x\in M\,|\, ||s_k(x)||<\nu\}$ of the submanifold $W_k$ does not intersect $\Sigma^\cF(f_k)$. Assume that $p\in U$. By the $\nu$--transversality along $\cF$ of the section $s_k$, 
 there is a unitary vector field $v\in T_p\cF$ such that $|| \nabla_v s_k(p)|| \geq \nu$. By asymptotic holomorphicity, for $k$ large, we have that  the unitary vector field $Jv \in T_p \cF$ satisfies $|| \nabla_{Jv} s_k(p) - i \nabla_{v} s_k(p)|| =O(k^{-1/2})$. Therefore, the map $\nabla^{\cF} s_k(p)$ is surjective. We conclude that $p \not \in  \Sigma^\cF(f_k)$. 
\end{proof}

\begin{lemma} \label{lem:morse}
Let $F$, a leaf of $\cF$, be fixed. After a perturbation of the sequence $s_k$, preserving transversality
to zero and approximately holomorphicity, it can be assumed that:
\begin{enumerate}
\item the restrictions of the $f_k$ to $F$ are Morse functions.
\item $\Sigma^\cF(f_k)$ is a finite union of disjoint circles in general position with respect to 
$\cF$. Their tangency points are turning points, i.e., birth-death type singularities for the restriction of $f_k$ to the corresponding leaf.
\end{enumerate}

\end{lemma}
\begin{proof}
According to \cite{FW86}, after an arbitrarily small  $C^r$ perturbation, $r\geq 2$, the set of tangential critical points $\Sigma^\cF(f_k)$
can be assumed to fit into a 1-dimensional manifold that is transverse to $\cF$ everywhere but at the finite collection of turning points $c_1,\dots,c_d$. Every other point is a non--degenerate critical point for the restriction of $f_k$ to the corresponding leaf. The turning points satisfy the following relevant property: in a small foliated chart, a plaque not containing the turning point intersects $\Sigma^\cF(f_k)$ either in the empty set or in two tangential critical points. 

Assertion (1) in the Lemma follows by showing that none of the 
$c_1,\dots,c_d$ belong to the fixed leaf $F$: if any of them do, a $C^r$--small isotopy, transverse to $\cF$ at the turning point, can be used to move it to a nearby leaf. This is described in detail in \cite{FW86}.

These $C^r$ perturbations of $f_k$ can be taken to be the result of a $C^r$ perturbation of $s_k$. 
Indeed, let $\varepsilon_k$ be a $C^r$ perturbation of $f_k$. The function $\varepsilon_k$ can be assumed to be identically zero away from an arbitrary small neighborhood of $\Sigma^\cF(f_k)$ so, by lemma \ref{lem:neigh}, the following expression is well defined:
\[ \widetilde{s}_k = s_k\sqrt{1 + \varepsilon_k/f_k} ,\]
since $f_k$ is bounded from below in the support of $\varepsilon_k$. It is clear that
\[ || \widetilde{s}_k || = f_k + \varepsilon_k. \]
The asymptotic holomorphicity of the sequence $\tilde{s}_k$ can be readily checked:
\[ \nabla \widetilde{s}_k = \nabla s_k\sqrt{1 + \varepsilon_k/f_k} + s_k\dfrac{f_k\nabla \varepsilon_k - \varepsilon_k \nabla f_k}{2f_k^2\sqrt{1 + \varepsilon_k/f_k}},  \]
where the second term is $C^r$-small and the first is $C^r$-close to $\nabla s_k$. A similar computation for the higher order derivatives concludes the claim. 
\end{proof}

We can finally address the proof of the theorem.

\begin{proof}[Proof of Theorem \ref{thm:main}]
Fix a leaf $F$ and assume that we have all the data needed for developing approximately--holomorphic geometry in $M^{2n+1}$. The metrics $g_k$ induce complete metrics in $F$. Given an approximately--holomorphic sequence $s_k$, with corresponding Donaldson-type submanifolds $W_k$, an application of Lemma \ref{lem:morse} yields a new approximately--holomorphic sequence, still denoted by $s_k$, that induces Morse functions $(f_k)_{\mid F}$ in $F \setminus W_k$.

By Lemma \ref{lem:neigh}, $W_k$ has an $\varepsilon$-neighborhood containing a regular level $a_k$. Lemmata \ref{lem:neigh} and \ref{lem:morse} together mean that $\Sigma^\cF(f_k)$ has a small tubular neighborhood of positive radius not intersecting the level $a_k$. 

By Lemma \ref{lem:morse}, the manifold $\Sigma^\cF(f_k)$ is transverse to $\cF$ except in a finite number of turning points $c_1,\dots,c_d$. Fix a closed geodesic arc $T_i$ through each $c_i$, transverse to the foliation. Let $B^{2n}(0,r) \subset \mathbb{R}^{2n}$ be the closed ball of radius $r$. For $r>0$ sufficiently small, the exponential map for the leafwise metric $g_k^\cF$ yields disjoint foliated charts $\phi_i: U_i  \rightarrow [0,1] \times B^{2n}(0,r)$ satisfying $\phi_i(T_i) = [0,1] \times \{0\}$. Having fixed $r$, by taking the $T_i$ sufficiently short -- effectively shrinking $U_i$ in the vertical direction -- it can be assumed that:
\[ \phi_i(\Sigma^\cF(f_k) \cap U_i) \subset [0,1] \times B^{2n}(0,r/2) \]

Consider the family of open arcs $I_j \cong (0,1) \subset \Sigma^\cF(f_k)$, $j \in [1,2,..,l]$, and circles $I_j \cong \mathbb{S}^1 \subset \Sigma^\cF(f_k)$, $j \in [l+1,2,..,m]$, comprising $\Sigma^\cF(f_k) \setminus (\cup_{i=1..d} U_i)$. For sufficiently small $0 < s < r$, the exponential map for the metric $g_k^\cF$ defines disjoint charts $\psi_j: V_j  \rightarrow I_j \times B^{2n}(0,s)$. The union of the $U_i$ and the $V_j$ covers $\Sigma^\cF(f_k)$.

The subsets $\mathcal{C}_i$, as in Proposition \ref{pro:dich}, can be defined and they come in two families:
\begin{enumerate}
\item $s/2$--neighborhoods, in the metric $g_k^\cF$, of the points $x \in I_j \cap F$, for any $j$,
\item $r/2$--neighborhoods, in the metric $g_k^\cF$, of the points $x \in T_i \cap F$, for any $i$.
\end{enumerate}

By construction, the $g_k^\cF$--diameter of the $\mathcal{C}_i$ is bounded above by $r/2$. Further,
the $g_k^\cF$--distance between any two sets $\mathcal{C}_i$ and $\mathcal{C}_{i'}$ is bounded below by $s$.
Therefore conditions (1) and (2) in Proposition \ref{pro:dich} hold. Condition (3) follows immediately from the fact that the union of the $\mathcal{C}_i$ is the intersection of a neighborhood of $\Sigma^\cF(f_k)$ with the leaf $F$.

An application of Lemma \ref{lem:dich} shows that the relative homotopy groups $\pi_j(F, F \cap W_k)$ 
vanish for $j < n$ and for $k$ large enough, since we already did the index computation in Proposition \ref{propo:index}.
\end{proof}


\begin{thebibliography}{xxxxx}
\bibitem{CM} Crainic, M.; Moerdijk, I.  Foliation groupoids and their cyclic homology. Adv. Math. 157 (2001), no. 2, 177--197.
\bibitem{Do96} Donaldson, S. K. Symplectic submanifolds and almost-complex geometry. J. Differential Geom. 44 (1996): 666--705.
\bibitem{FW86} Ferry, S.C.; Wasserman, A. G. Morse theory for codimension-one foliations. Trans. Amer. Math. Soc. 298 (1986), no. 1, 227--240.
\bibitem{Ha02} Haefliger, A. Foliations and compactly generated pseudogroups. Foliations:
 geometry and dynamics (Warsaw, 2000), 275--295, World Sci. Publ., River Edge, NJ, 2002.
 \bibitem{Hat} Hatcher, A.  Algebraic topology. 
Cambridge: Cambridge University Press (ISBN 0-521-79540-0/pbk). xii, 544 p. (2002).  {\em https://www.math.cornell.edu/\textasciitilde hatcher/AT/AT.pdf}
\bibitem{IM04a} Ibort, A.;   Mart\'inez Torres, D. Approximately holomorphic geometry and estimated transversality on 2-calibrated manifolds. C. R. Math. Acad. Sci.
Paris 338, no. 9 (2004): 709--712.
\bibitem{Ma12} Mart\'inez Torres, D.  Codimension-one foliations calibrated by nondegenerate closed 2-forms. Pacific J. Math. 261 (2013), no. 1, 165--217.
\bibitem{Me12} Meigniez, G. Meigniez, Gaël. Regularization and minimization of codimension-one Haefliger structures. J. Differential Geom. 107 (2017), no. 1, 157--202.
\bibitem{Mil}  Milnor, J. Morse theory. Annals of Mathematics Studies, No. 51 Princeton University Press, Princeton, N.J. 1963
\bibitem{MM03} Moerdijk, I.; Mrcun, J. Introduction to foliations and Lie groupoids. 
Cambridge Studies in Advanced Mathematics, 91. Cambridge University Press, Cambridge, 2003.
\bibitem{Su79} Sullivan, D. A homological characterization
  of foliations consisting of minimal surfaces. Comment. Math. Helv. 54 (1979), no. 2, 218--223. 

\end{thebibliography}
\end{document}